\documentclass[12pt,a4paper]{amsart}
\usepackage[top=35mm, bottom=32mm, left=30mm, right=30mm]{geometry}
\usepackage{mathptmx}
\usepackage{mathrsfs}

\usepackage{verbatim}
\usepackage{url}
\usepackage[all]{xy}
\usepackage{xcolor}
\usepackage[colorlinks=true,citecolor=blue]{hyperref}

\theoremstyle{plain}
\newtheorem{claim}{Claim}
\newtheorem{thm}{Theorem}[section]
\newtheorem{lem}[thm]{Lemma}
\newtheorem{prop}[thm]{Proposition}
\newtheorem{cor}[thm]{Corollary}

\theoremstyle{definition}

\newtheorem{exmp}[thm]{Example}
\newtheorem{rem}[thm]{Remark}
\newtheorem{ques}[thm]{Question}

\newcommand{\Z}{{\mathbb{Z}_+}}
\newcommand{\N}{\mathbb{N}}

\newcommand{\mZ}{\mathbb Z}

\newcommand{\ep}{\varepsilon}

\DeclareMathOperator{\diam}{diam}
\DeclareMathOperator{\supp}{supp}

\begin{document}
\title{Mean equicontinuity and mean sensitivity}
\author[J. Li]{Jian Li}
\address[J.~Li]{Department of Mathematics, Shantou University, Shantou, Guangdong 515063, P.R. China}
\email{lijian09@mail.ustc.edu.cn}

\author[S.~Tu]{Siming Tu}
\address[S.~Tu]{Wu Wen-Tsun Key Laboratory of Mathematics, USTC, Chinese Academy of Sciences and
School of Mathematics, University of Science and Technology of China,
Hefei, Anhui, 230026, P.R. China}
\email{tsming@mail.ustc.edu.cn}

\author[X.~Ye]{Xiangdong Ye}
\address[X.~Ye]{Wu Wen-Tsun Key Laboratory of Mathematics, USTC, Chinese Academy of Sciences and
School of Mathematics, University of Science and Technology of China,
Hefei, Anhui, 230026, P.R. China}
\email{yexd@ustc.edu.cn}

\date{\today}

\begin{abstract} Answering an open question affirmatively it is shown that
every ergodic invariant measure of a mean equicontinuous (i.e. mean-L-stable) system has discrete spectrum.
Dichotomy results related to mean equicontinuity and mean sensitivity are obtained when a dynamical
system is transitive or minimal.

Localizing the notion of mean equicontinuity,
notions of almost mean equicontinuity and almost Banach mean equicontinuity are introduced.
It turns out that a system with the former property may have positive entropy and meanwhile a
system with the later property must have zero entropy.
\end{abstract}
\maketitle

\section{Introduction}
Let $X$ be a compact metric space with a metric $d$, and let $T$ be a continuous map from $X$ to itself.
The pair $(X,T)$ will be called a \emph{(topological) dynamical system}.

A dynamical system $(X,T)$ is called \emph{equicontinuous} if for every $\ep>0$ there is a $\delta>0$ such that
whenever $x,y\in X$ with $d(x,y)<\delta$, $d(T^nx,T^ny)<\ep$ for $n=0,1,2,\dotsc$,
that is, the family of maps $\{T^n\colon n\in \Z\}$ is uniformly equicontinuous.
Equicontinuous systems have simple dynamical behaviors.
It is well known that a dynamical system $(X,T)$ with $T$ being surjective is equicontinuous if and only if
there exists a compatible metric $\rho$ on $X$ such that $T$ acts on $X$ as an isometry, i.e.,
$\rho(Tx,Ty)=\rho(x,y)$ for any $x,y\in X$.
Moreover, a transitive equicontinuous system is conjugate to
a minimal rotation on a compact abelian metric group,
and $(X,T,\mu)$ has discrete spectrum, where $\mu$ is the unique Haar measure on $X$.

When studying dynamical systems with discrete spectrum, Fomin~\cite{F51} introduced a notion
called \emph{stable in the mean in the sense of Lyapunov} or simply \emph{mean-L-stable}.
A dynamical system $(X,T)$ is \emph{mean-L-stable}
if for every $\ep>0$, there is a $\delta>0$ such that $d(x,y)<\delta$
implies $d(T^nx,T^ny)<\ep$ for all $n\in\Z$ except a set of
upper density less than $\ep$.
Fomin proved that if a minimal system is mean-L-stable then it is uniquely ergodic.
Mean-L-stable systems are also discussed briefly by Oxtoby in~\cite{O52},
and he proved that each transitive mean-L-stable system is uniquely ergodic.
Auslander in~\cite{A59} systematically studied mean-L-stable systems, and provided new examples.
See Scarpellini \cite{S82} for a related work.
It is an open questions if every ergodic invariant measure on a mean-L-stable system has discrete spectrum~\cite{S82}.
We will give an affirmative answer to this question (Theorem \ref{thm:mean-eq-discrete-spectrum}).

We introduce equicontinuity in the mean sense, more precisely,
a dynamical system $(X,T)$ is called \emph{mean equicontinuous} if for every $\ep>0$,
there exists a $\delta>0$ such that whenever $x,y\in X$ with $d(x,y)<\delta$,
\[\limsup_{n\to\infty}\frac{1}{n}\sum_{i=0}^{n-1}d(T^ix,T^iy)<\ep.\]
We show that a dynamical system is mean equicontinuous if and only if it is mean-L-stable
(Lemma~\ref{lem:mean-equi-mean-L-stalbe}).
We also show each dynamical system admits a maximal mean equicontinuous factor (Theorem \ref{thm:max-mean-equi-factor}),
and mean equicontinuity is preserved by factor maps (Theorem \ref{thm:mean-equi-factor}).
We remark that studying dynamical properties in the mean sense is an
interesting topic, see \cite{OW} for the research on mean distality
and \cite{tomasz, HLY13} for the investigation on mean Li-Yorke chaos.

The notion of sensitivity was introduced when studying the complexity of a dynamical system,
and it is a part of the known definition of chaos in the Devaney sense.
We say that a dynamical system $(X,T)$ has \emph{sensitive dependence on initial condition}
or briefly $(X,T)$ is \emph{sensitive}
if there exists a $\delta>0$ such that for every $x\in X$
and every neighborhood $U$ of $x$, there exists $y\in U$ and $n\in\N$ such that $d(T^nx,T^ny)>\delta$.

When considering the opposite side of sensitivity the notion of equicontinuity at a point appears naturally, see \cite{GW93}.
That is a point $x\in X$ is called an \emph{equicontinuous point} (or $(X,T)$ is \emph{equicontinuous at $x$})
if for every $\ep>0$ there is a $\delta>0$ such that for every $y\in X$ with $d(x,y)<\delta$,
$d(T^nx,T^ny)<\ep$ for all $n\in\Z$.
If every point in $X$ is an equicontinuous point then by the compactness of $X$
the dynamical system $(X,T)$ is equicontinuous. A transitive system is called
\emph{almost equicontinuous} if there is at least one equicontinuous point.
Almost equicontinuous systems have been studied intensively and
have many applications. For example, the enveloping semigroup $E(X)$ is metrizable if and only if
$(X,T)$ is hereditarily almost equicontinuous \cite{GMU}.

We know that if $(X,T)$ is almost equicontinuous then the set of equicontinuous points coincides
with the set of all transitive points \cite{AAB}, it is uniformly rigid ~\cite{GW93} and thus has zero topological entropy ~\cite{GM89}.
We have the following dichotomy results.
If $(X,T)$ is minimal, then $(X,T)$ is either equicontinuous  or sensitive \cite{AY80};
and if $(X,T)$ is transitive, then $(X,T)$ is either almost equicontinuous or sensitive \cite{AAB}.

\medskip

Inspirited by the above ideas, we will introduce  notions of almost mean equicontinuity and mean sensitivity.
A point $x\in X$ is called \emph{mean equicontinuous} if for every $\ep>0$,
there exists a $\delta>0$ such that for every $y\in X$ with $d(x,y)<\delta$,
\[\limsup_{n\to\infty}\frac{1}{n}\sum_{i=0}^{n-1}d(T^ix,T^iy)<\ep.\]
A transitive system is called \emph{almost mean equicontinuous} if there is at least one mean equicontinuous point.
A dynamical system $(X,T)$ is called \emph{mean sensitive}
there exists a $\delta>0$ such that for every $x\in X$
and every neighborhood $U$ of $x$, there exists $y\in U$ and $n\in\N$ such that
\[\limsup_{n\to\infty}\frac{1}{n}\sum_{i=0}^{n-1}d(T^ix,T^iy)>\delta.\]

We show that if a dynamical system $(X,T)$ is minimal,
then $(X,T)$ is either  mean equicontinuous or mean sensitive (Corollary \ref{meorms}),
and if $(X,T)$ is transitive, then $(X,T)$ is either  almost mean equicontinuous or mean sensitive
(Theorem \ref{almostmeorms}).
Unlike the case of almost equicontinuous systems,
we show that for almost mean equicontinuous systems the set of transitive points is contained in the
set of all mean equicontinuous points and there are examples in which they do not coincide.
It is unexpected that there are almost
mean equicontinuous systems admitting positive topological entropy (Theorem \ref{p-entropy}),
while every almost equicontinuous system has zero topological entropy.

Thus it is natural to seek a class of mean equicontinuous systems for which the localized systems at least have zero entropy.
We find the class of {\it Banach mean equicontinuous systems} obtained by replacing small upper density with small
Banach density in the definition of mean-L-stable systems is the right one.
Namely we show that almost Banach mean equicontinuous systems
have zero topological entropy (Corollary \ref{entropyzero}),
and this implies that the almost mean equicontinuous systems admitting positive topological entropy
we constructed in Theorem \ref{p-entropy} are not almost Banach mean equicontinuous.
The deep reason of this is that in a transitive system
a transitive point can only approach a ``chaotic subsystem''
for time segments, which may result large Banach density and at the same time small upper density.

\medskip

\section{Preliminaries}
In this section we recall some notions and aspects of the theory of topological dynamical systems.

\subsection{Subsets of non-negative integers}
Denote by $\Z$ ($\N$, $\mZ$, respectively) the set of all non-negative integers (positive integers, integers, respectively).
Let $F\subset \Z$. We say that $F$ is an \emph{IP-set} if there is a subsequence $\{p_i\}$ of $\N$ such that
\[\{p_{i_1}+\dotsb+p_{i_k}\colon i_1<\dotsb<i_k, n\in\N\}\subset F;\]
$F$ is \emph{syndetic} if there is $k > 0$ such that
$[i,i+k]\cap F\neq\emptyset$ for every $i \in \N$.

We define the \emph{upper density} $\overline{D}(F)$ of $F$ by
\[ \overline{D}(F)=\limsup_{n\to\infty} \frac{\#(F\cap[0,n-1])}{n},\]
where $\#(\cdot)$ is the number of elements of a set.
Similar, $\underline{D}(F)$, the \emph{lower density} of $F$, is defined by
\[ \underline{D}(F)=\liminf_{n\to\infty} \frac{\#(F\cap[0,n-1])}{n}.\]
One may say $F$  has \emph{density} $D(F)$ if $\overline{D}(F)=\underline{D}(F)$,
in which case $D(F)$ is equal to this common value.
The \emph{upper Banach density} $BD^*(F)$ is defined by
\[BD^*(F)=\limsup_{N-M\to\infty} \frac{\#(F\cap[M,N])}{N-M+1}.\]
Similarly, we can define the \emph{lower Banach density} $BD_*(F)$ and \emph{Banach density} $BD(F)$.

\subsection{Compact metric spaces}
Let $(X,d)$ be a compact metric space.
For $x\in X$ and $\ep>0$, denote $B(x,\ep)=\{y\in X\colon d(x,y)<\ep\}$.
Denote by the product space $X\times X=\{(x,y)\colon x,y\in X\}$ and the diagonal $\Delta_X=\{(x,x)\colon x\in X\}$.
A subset of $X$ is called a \emph{$G_\delta$ set}
if it can be expressed as a countable intersection of open sets;
a \emph{residual set} if it contains the intersection of a countable collection of dense open sets.
By the Baire category theorem, a residual set is also dense in  $X$.

Let $C(X)$ be the set of continuous real functions on $X$  with the supremum norm
$\Vert f\Vert=\sup_{x\in X}|f(x)|$.
Let $M(X)$ be the set of regular Borel probability measures on $X$.
The \emph{support} of a measure $\mu\in M(X)$, denoted by $\supp(\mu)$,
is the smallest closed subset $C$ of $X$ such that $\mu(C)=1$.
We regard $M(X)$ as a closed convex subset of $C(X)^*$, the dual space of $C(X)$, equipped with the
weak$^*$ topology. Then $M(X)$ is a compact metric space.

\subsection{Topological dynamics}
Let $(X,T)$ be a dynamical system.
The \emph{orbit} of a point $x\in X$, $\{x,Tx,T^2x, \ldots,\}$, is denoted by $Orb(x,T)$.
The $\omega$-limit set of $x$ is the set of limit points of the orbit sequence
\[\omega(x,T)=\bigcap_{N\geq 0} \overline{\{T^nx\colon n\geq N\}}.\]

If $A$ is a non-empty closed subset of $X$ and $TA\subset A$, then $(A,T|_A)$ is called a \emph{subsystem} of $(X,T)$,
where $T|_A$ is the restriction of $T$ on $A$. If there is no ambiguity, we will use the notation $T$ instead of $T|_A$.

We say that a point $x\in X$ is \emph{recurrent} if $x\in\omega(x,T)$.
The system $(X,T)$ is called \emph{(topologically) transitive} if $\omega(x,T)=X$ for some $x\in X$, and
such a point $x$ is called a \emph{transitive point}. Denote by $Trans(X,T)$ the set of transitive points of $(X,T)$.
With a Baire category argument, one can show that
if $(X,T)$ is transitive then $Trans(X,T)$ is a dense $G_\delta$ subset of $X$.
If the product system $(X\times X,T\times T)$ is transitive, then we say that $(X,T)$ is \emph{weakly mixing}.

The system $(X,T)$ is said to be \emph{minimal} if every point of $X$ is a transitive point (i.e., $Trans(X,T)=X$).
A subset $Y$ of $X$ is called \emph{minimal} if $(Y,T)$ forms a minimal subsystem of $(X,T)$.
A point $x\in X$ is called \emph{minimal} if it is contained in a minimal set $Y$ or,
equivalently, if the subsystem $(\overline{Orb(x,T)},T)$ is minimal.

For $x\in X$ and $A\subset X$, let $N(x,A)=\{n\in\Z\colon T^nx\in A\}$.
If $U$ is a neighborhood of $x$, then the set $N(x,U)$ is called the set of \emph{return times} of the point $x$
to the neighborhood $U$.
The following result is well-known, see~\cite{F81} for example.
\begin{lem} \label{lem:rec-ip}
Let $(X,T)$ be a dynamical system and $x\in X$. Then
\begin{enumerate}
  \item $x$ is recurrent if and only if $N(x,U)$ contains an IP-set for every neighborhood $U$ of $x$;
  \item $x$ is minimal if and only if $N(x,U)$ is syndetic for every neighborhood $U$ of $x$.
\end{enumerate}
\end{lem}

A pair of points $(x,y)\in X\times X$ is said to be \emph{proximal}
if for any $\ep>0$, there exists a positive integer $n$ such that $d(T^nx,T^ny)<\ep$.
Let $P(X,T)$ denote the collection of all proximal pairs in $(X,T)$.
The dynamical system $(X,T)$ is called \emph{proximal} if any pair of two points in $X$ is proximal,
i.e., $P(X,T)=X\times X$.
If $(x,y)\in X\times X$ is not proximal, then it is said to be \emph{distal}.
A dynamical system $(X,T)$ is called \emph{distal} if any pair
of distinct points in $(X,T)$ is distal.

A pair of points $(x,y)\in X\times X$ is said to be \emph{Banach proximal}
if for any $\ep>0$, $d(T^nx,T^ny)<\ep$ for all $n\in \Z$ except a set of zero Banach density.
Let $BP(X,T)$ denote the collection of all Banach proximal pairs in $(X,T)$.
See \cite{LT13} for a detailed study on Banach proximality.

Recall that a pair of points $(x,y)$ is called \emph{regionally proximal}
if for every $\ep>0$, there exist two points  $x',y'\in X$
with $d(x,x')<\ep$ and $d(y,y')<\ep$, and a positive integer $n$ such that $d(T^nx',T^ny')<\ep$.
Let $Q(X,T)$ be the set of all regionally proximal pairs in $(X,T)$.
Clearly, $Q(X,T)\supset P(X,T)\supset BP(X,T)$.

When $(X,T)$ and $(Y,S)$ are two dynamical systems and $\pi\colon X\to Y$ is a continuous onto map which
intertwines the actions (i.e., $\pi\circ T=S\circ \pi$),
one says that $(Y, S)$ is a \emph{factor} of $(X,T )$ or $(X, T)$ is an \emph{extension} of $(Y, S)$,
and $\pi$ is a \emph{factor map}.
If $\pi$ is a homeomorphism, then we say that $\pi$ is a \emph{conjugacy} and
that the dynamical systems $(X,T)$ and $(Y,S)$ are \emph{conjugate}.
By the Halmos and von Neumann Theorem (see \cite[Theorem 5.18]{W82}),
a minimal system  is equicontinuous if and only if it is
conjugate to a minimal rotation on a compact abelian metric group.

If $\pi\colon (X,T)\to (Y,S)$ is a factor map, then $R_\pi=\{(x,x')\in X\times X\colon \pi(x)=\pi(x')\}$
is closed $T\times T$-invariant equivalence relation, that is $R_\pi$ is a closed subset of $X \times X$ and
if $(x,x')\in R_\pi$, then $(Tx,Tx')\in R_\pi$.
Conversely, if $R$ is a closed $T\times T$-invariant equivalence relation on $X$,
then the quotient space $X/R$ is a compact metric space
and $T$ naturally induces an action on $X/R$ by $T_R([x])=[Tx]$.
Then $(X/R,T_R)$ forms a dynamical system and the quotient map $\pi_R\colon X\to X/R$ is a factor map.
Hence there is a one-to-one correspondence between factors and closed invariant equivalence relations,
we will use them interchangeably.
A factor map  $\pi\colon (X,T)\to (Y,S)$ is called \emph{proximal} (resp. \emph{Banach proximal})
 if whenever $\pi(x)=\pi(y)$ the pair $(x,y)$ is proximal (resp. Banach proximal).

An equicontinuous factor of $(X,T)$
is maximal if any other equicontinuous factor of $(X,T)$ factors through it.
It is thus unique up to conjugacy and therefore referred to as the \emph{maximal equicontinuous factor}.
Let $\pi\colon(X,T)\to (Y,S)$ be the factor map to the maximal equicontinuous factor.
The equivalence relation $R_\pi$ is called the \emph{equicontinuous structure relation}.
Similarly, we can define \emph{the maximal distal factor} and the \emph{distal structure relation}.
It is shown in~\cite{EG60} that the equicontinuous structure relation is the smallest closed invariant equivalence relation
containing the regional proximal relation, and the distal structure relation
is the smallest closed invariant equivalence relation
containing the proximal relation.

We refer the reader to the textbook~\cite{W82} for information on topological entropy.
\subsection{Invariant measures}
Let $(X,T)$ be a dynamical system and $M(X,T)$ be the set of $T$-invariant regular Borel probability measures on $X$.
It is well known that any dynamical system $(X,T)$ admits at least one
$T$-invariant regular Borel probability measures. Moreover, we known that $M(X,T)$ is a compact metric space.
An invariant measure is ergodic if and only if it is an extreme point of $M(X,T)$.
The \emph{support} of a dynamical system $(X,T)$, denoted by $\supp(X,T)$,
is the smallest closed subset $C$ of $X$ such that $\mu(C)=1$ for all $\mu\in M(X,T)$.

The action of $T$ on $X$ induces an action on
$M(X)$ in the following way: for $\mu\in M(X)$ we define $T\mu$ by
\[\int_X f(x)\ dT\mu(x)= \int_X f(Tx)d\mu(x), \quad \forall f\in C(X).\]
Hence $(M(X),T)$ is also a topological dynamical system.
We say $(X,T)$ is \emph{strongly proximal} if $(M(X),T)$ is proximal.

Let $(X,T)$ be a dynamical system.
We call $(X,T)$ an \emph{$E$-system} if it is transitive and there exists $\mu\in M(X,T)$ such that $\supp(\mu)=X$.
We say that $(X,T)$ is \emph{uniquely ergodic} if $M(X,T)$ consists a single measure.
If $(X,T)$ is a minimal rotation on a compact abelian metric group, then it is uniquely ergodic and
the Haar measure is the only invariant measure.

For a dynamical system $(X,T)$, $f\in C(X)$ and $n\in\N$, let
$f_n(x)=\frac{1}{n}\sum_{i=0}^{n-1} f(T^ix).$ The following theorem is well known.

\begin{thm}[\cite{O52}] \label{thm:unique-ergodic}
Let $(X,T)$ be a dynamical system. Then the following conditions are equivalent:
\begin{enumerate}
  \item $(X,T)$ is uniquely ergodic;
  \item for each $f\in C(X)$, $\{f_n\}_{n=1}^\infty$ converges uniformly on $X$ to a constant;
\item for each $f\in C(X)$, there is a subsequence $\{f_{n_k}\}_{k=1}^\infty$
which converges pointwise on $X$ to a constant.
\item $(X,T)$ contains only one minimal set, and
for each $f\in C(X)$, $\{f_n\}_{n=1}^\infty$ converges uniformly on $X$.
\end{enumerate}
\end{thm}

Let $(X,T)$ be a dynamical system and $\mu$ be an invariant measure.
A complex number $\lambda$ is called an \emph{eigenvalue} of $(X,T,\mu)$ if there is $f\in L^2(\mu)$, with
$f$ not the zero function, satisfying $f(Tx)=\lambda f(x)$ for $\mu$-a.e. $x\in X$.
The function $f$ is called an \emph{eigenfunction} of $(X,T,\mu)$ corresponding the eigenvalue $\lambda$.
The measurable dynamical system $(X,T,\mu)$ has \emph{discrete spectrum} or \emph{pure point spectrum} if
there exists an orthonormal basis for $L^2(\mu)$ which consists of eigenfunctions of $(X,T,\mu)$.
If $\mu$ is ergodic, then $(X,T,\mu)$ has discrete spectrum if and only if it is conjugate to
an ergodic rotation on some compact abelian group (see \cite[Theorem~3.6]{W82}).

\section{Mean equicontinuous systems}
In this section, we study mean equicontinuous systems.
We obtain several equivalent conditions of mean equicontinuity.
We also show that every ergodic invariant measure on a mean equicontinuous system has discrete spectrum,
every dynamical system admits a maximal mean equicontinuous factor,
and mean equicontinuity is preserved by factor maps. We first observe that mean equicontinuity
is equivalent to mean-L-stability which was first introduced in~\cite{F51}.

\begin{lem} \label{lem:mean-equi-mean-L-stalbe}
A dynamical system is mean equicontinuous if and only if it is mean-L-stable.
\end{lem}
\begin{proof}
Assume that $(X,T)$ is mean equicontinuous.
For every $\ep>0$ there exists a $\delta>0$ such that
\[\limsup_{n\to\infty}\frac{1}{n}\sum_{i=0}^{n-1}d(T^ix,T^iy)<\ep^2\]
for all $x,y\in X$ with $d(x,y)<\delta$.
Let $E=\{k\in\Z: d(T^kx,T^ky)\geq \ep\}$. One has
\begin{align*}
  \ep^2>\limsup_{n\to\infty} \frac{1}{n}\sum_{i=0}^{n-1}d(T^ix,T^iy)\geq
  \limsup_{n\to\infty} \frac{1}{n} (\ep\#([0,n-1]\cap E))=\ep\overline{D}(E),
    \end{align*}
and then $\overline{D}(E)<\ep$, which implies $(X,T)$ is mean-L-stable.

Conversely, assume that $(X,T)$ is mean-L-stable.
Fix a positive number $\ep>0$ and choose a positive number $\eta<\frac{\ep}{\diam(X)+1}$.
There is a $\delta>0$ such that $d(x,y)<\delta$
implies $d(T^nx,T^ny)<\eta$ for all $n\in\Z$ except a set of
upper density less than $\eta$.
Let $F=\{k\in\Z: d(T^kx,T^ky)\geq \eta\}$. One has $\overline{D}(F)<\eta$ and
\begin{align*}
  \limsup_{n\to\infty} \frac{1}{n}\sum_{i=0}^{n-1}d(T^ix,T^iy)&\leq
  \limsup_{n\to\infty} \frac{1}{n} (\diam(X)\#([0,n-1]\cap F)+\eta n)\\
&\leq \diam(X)\overline{D}(F)+\eta<\ep,
   \end{align*}
 which implies $(X,T)$ is mean equicontinuous.
\end{proof}

The following lemma is easy to verify.
\begin{lem} \label{lem:mean-prod}
Let $(X,T)$ and $(Y,S)$ be two dynamical systems.
Then $(X\times Y,T\times S)$ is mean equicontinuous if and only if
both $(X,T)$ and $(Y,S)$ are mean equicontinuous.
\end{lem}

We have the following characterization of mean equicontinuous systems which is implicit in~\cite{O52}.
For completeness we include a proof.
\begin{thm}\label{thm:mean-equi-unifon-conv}
Let $(X,T)$ be a dynamical system. Then the following conditions are equivalent:
\begin{enumerate}
  \item $(X,T)$ is mean equicontinuous;
  \item for each $f\in C(X\times X)$, the sequence $\{f_n\}_{n=1}^\infty$ is uniformly equicontinuous;
  \item for each $f\in C(X\times X)$, the sequence $\{f_n\}_{n=1}^\infty$ is uniformly convergent to a
  $T\times T$-invariant continuous  function $f^*\in C(X\times X)$.
\end{enumerate}
\end{thm}
\begin{proof}
To make the idea of the proof clearer, when proving (1)$\Rightarrow$(2) and (2)$\Rightarrow$(3),
we assume $f\in C(X)$ instead of $f\in C(X\times X)$, because if $(X,T)$ is mean equicontinuous
then so is $(X\times X,T\times T)$.

(1)$\Rightarrow$(2)
Fix $f\in C(X)$. To show that $\{f_n\}_{n=1}^\infty$ is uniformly equicontinuous, it suffices to show that
for any $\ep>0$, there exists $\delta>0$ such that $x$, $y\in X$ with $d(x,y)<\delta$
implies $|f_n(x)-f_n(y)|<\ep$ for all $n\in\N$. By the definition of $f_n$, one has
\begin{align*}
|f_n(x)-f_n(y)|= \Bigl\vert\frac{1}{n}\sum_{i=0}^{n-1} f(T^ix)-\frac{1}{n}\sum_{i=0}^{n-1} f(T^iy)\Bigr\vert
\leq \frac{1}{n} \sum_{i=0}^{n-1}\bigl\vert f(T^ix)-f(T^iy)\bigr\vert.
\end{align*}

Fix a positive number $\ep>0$. By uniform continuity of $f$,
there exists $\delta_1>0$ such that if $x$, $y\in X$ with $d(x,y)<\delta_1$
then $|f(x)-f(y)|<\frac{\ep}{2}$.
As $(X,T)$ is mean equicontinuous, it is also mean-L-stable.
Choose a positive number $\eta<\min\{\delta_1,\frac{\epsilon}{8\Vert f\Vert+1}\}$.
There is $\delta_2\in(0,\delta_1)$ such that $d(x,y)<\delta_2$
implies $d(T^nx,T^ny)<\eta$
for all $n\in\Z$ except in a set $E$ with $\overline{D}(E)<\eta$.

Choose $N$ large enough such that $\frac{1}{n}\#(E\cap[0,n-1])<2\eta$ for all $n\geq N$.
Then for every $n\geq N$ and $x,y\in X$ with $d(x,y)<\delta_2$, one has

\begin{align*}
 \frac{1}{n} \sum_{i=0}^{n-1}|f(T^ix)-f(T^iy)| & \leq
 \frac{1}{n}\Bigl(\sum_{i\in E\cap [0,n-1]} 2 \Vert f\Vert+
 \sum_{i\in [0,n-1]\setminus E}\bigl\vert f(T^ix)-f(T^iy)\bigr\vert\Big)\\
 &\leq 4\eta \Vert f\Vert +\frac{\ep}{2}<\ep.
\end{align*}

By the compactness of $X$, there exists $\delta_3>0$ such that
for every $n\in\{1,2,\dotsc,N\}$ and $x,y\in X$ with $d(x,y)<\delta_3$,
$|f_n(x)-f_n(y)|<\ep$.
Choose $0<\delta<\min\{\delta_2,\delta_3\}$.
Then for every $n\in\N$ and $x,y\in X$ with $d(x,y)<\delta$,
$|f_n(x)-f_n(y)|<\ep$.
This shows that $\{f_n\}_{n=1}^\infty$ is uniformly equicontinuous.

(2)$\Rightarrow$(3)
Clearly, $\Vert f_n\Vert\leq \Vert f\Vert$ for every $n\in\N$.
Then the sequence $\{f_n\}_{n=1}^\infty$ is uniformly bounded.
Since $\{f_n\}_{n=1}^\infty$ is uniformly equicontinuous,
by the Arzela-Ascoli Theorem there exists a subsequence $\{f_{n_k}\}_{k=1}^\infty$
which is uniformly convergent to $f^*\in C(X)$.
Then $f^*$ is $T$-invariant, that is
for every $x\in X$
\begin{align*}
\vert f^*(x)-f^*(Tx)\vert &=\Bigl\vert\lim_{k\to\infty}
\frac{1}{n_k}\Bigl(\sum_{i=0}^{n_k-1}f(T^ix)-\sum_{i=0}^{n_k-1}f\bigl(T^i (Tx)\bigr)\Bigr)\Bigr\vert\\
&=\lim_{k\to\infty} \frac{1}{n_k}\bigl\vert f(T^{n_k}x)-f(x)\bigr\vert=0
\end{align*}
By the continuity of $f^*$, $f^*|_{\overline{Orb(x,T)}}$ is constant for every $x\in X$.
Then by Theorem~\ref{thm:unique-ergodic} $(\overline{Orb(x,T)},T)$ is uniquely ergodic.
Again by Theorem~\ref{thm:unique-ergodic} $f_n\to f^*$ as $n\to\infty$ uniformly on  $(\overline{Orb(x,T)},T)$.
Since $x$ is arbitrary, one has $f_n(x)\to f^*(x)$ as $n\to\infty$ pointwise for every $x\in X$.
Then $f_n(x)\to f^*(x)$  as $n\to\infty$ uniformly on $X$, since  $\{f_n\}_{n=1}^\infty$ is uniformly equicontinuous.

(3)$\Rightarrow$(1)
Recall that $d(\cdot,\cdot)$ is the metric on $X$ and is a continuous function on $X\times X$.
Then the sequence $\{d_n\}_{n=1}^\infty$ is uniformly equicontinuous, that is
for every $\ep>0$ there exists $\delta>0$ such that $d_n(x,y)<\frac{\ep}{2}$ for every $n\in\Z$ and
every $x,y\in X$ with $d(x,y)<\delta$. Then for every $x,y\in X$ with $d(x,y)<\delta$, one has
\[\limsup_{n\to\infty}\frac{1}{n}\sum_{i=0}^{n-1}d(T^ix,T^iy)\leq \sup_{n\in\N} d_n(x,y)<\ep,\]
which implies that $(X,T)$ is mean equicontinuous.
\end{proof}

By Theorems~\ref{thm:unique-ergodic} and \ref{thm:mean-equi-unifon-conv}, we have the following corollary.
\begin{cor}[\cite{O52}] \label{cor:mean-eq-erg}
Let $(X,T)$ be a dynamical system.
If $(X,T)$ is mean equicontinuous, then for every $x\in X$, $(\overline{Orb(x,T)},T)$ is uniquely ergodic.
In particular, if $(X,T)$ is also transitive, then $(X,T)$ is uniquely ergodic.
\end{cor}

It is shown in~\cite{A59} that in a mean equicontinuous system a pair of points is proximal
if and only if it is persistently proximal (proximal with density one).
We can strengthen this result as follows.

\begin{thm}\label{thm:mean-eq-BP}
Let $(X,T)$  be a dynamical system. If $(X,T)$ is mean equicontinuous, then $Q(X,T)=P(X,T)=BP(X,T)$ and
it is a closed invariant equivalence relation.
\end{thm}

\begin{proof} It is clear that $Q(X,T)\supset P(X,T)\supset BP(X,T)$. We are going to show that
$P(X,T)=BP(X,T)$ and $Q(X,T)=P(X,T)$.

Assume first that $(x,y)\in P(X,T)$. Since $(X\times X,T\times T)$ is also mean-L-stable,
by Corollary~\ref{cor:mean-eq-erg} $\overline{Orb((x,y),T\times T)}$ is uniquely ergodic.
The proximality of $(x,y)$ implies that
$\overline{Orb((x,y),T\times T)}\cap \Delta_X\neq\emptyset$, which again implies that
$\supp(\overline{Orb((x,y),T\times T)})\subset \Delta_X$ and thus $(x,y)\in BP(X,T)$ by \cite{LT13}.

Assume now that $(x,y)\in Q(X,T)$. As $(X,T)$ is mean equicontinuous, for every $\ep>0$, there exists a $\delta>0$
such that whenever $z_1,z_2\in X$ with $d(z_1,z_2)<\delta$,
\[\limsup_{n\to\infty}\frac{1}{n}\sum_{i=0}^{n-1}d(T^iz_1,T^iz_2)<\frac{\ep}{3}.\]
For this $\delta>0$,
there exist $x',y'\in X$ and $k\in\N$ such that $d(x,x')<\delta$, $d(y,y')<\delta$ and
$d(T^k x',T^ky')<\delta$. Then
\begin{align*}
  \limsup_{n\to\infty}\frac{1}{n}\sum_{i=0}^{n-1}d(T^ix,T^iy)
&\leq   \limsup_{n\to\infty}\frac{1}{n}\sum_{i=0}^{n-1}(d(T^ix,T^ix')+d(T^iy',T^iy)+d(T^ix',T^iy'))\\
&\leq \limsup_{n\to\infty}\frac{1}{n}\sum_{i=0}^{n-1} d(T^ix,T^ix')+
\limsup_{n\to\infty}\frac{1}{n}\sum_{i=0}^{n-1}d(T^iy',T^iy)\\
&\qquad +\limsup_{n\to\infty}\frac{1}{n}\sum_{i=0}^{n-1} d(T^i(T^kx'),T^i(T^ky'))\\
&< \frac{\ep}{3}+\frac{\ep}{3}+\frac{\ep}{3}=\ep.
\end{align*}
Since $\ep$ is arbitrary, one has
\[\limsup_{n\to\infty}\frac{1}{n}\sum_{i=0}^{n-1}d(T^ix,T^iy)=0,\]
which implies that $(x,y)$ is proximal.

In general $Q(X,T)$ is a closed relation and $BP(X,T)$ is an invariant equivalence relation,
which implies $Q(X,T)$ is a closed invariant equivalence relation when $(X,T)$ is mean equicontinuous.
\end{proof}

Theorem~\ref{thm:mean-eq-BP} has the following direct corollaries.
\begin{cor}\label{cor:mean-eq}
Let $(X,T)$ be a dynamical system.
\begin{enumerate}
  \item If $(X,T)$ is mean equicontinuous, then
\begin{enumerate}
  \item it is a proximal extension of its maximal
equicontinuous factor;
\item the maximal distal factor and the maximal
equicontinuous factor coincide.
\end{enumerate}
\item If $(X,T)$ is mean equicontinuous and distal, then it is equicontinuous.
\end{enumerate}
\end{cor}

\begin{cor}\label{cor:mean-eq-wm}
Let $(X,T)$ be a dynamical system. Assume that $P(X,T)$ is dense in $X\times X$,
for example $(X,T)$ is weakly mixing or transitive with a fixed point \cite{HY02}.
Then it is mean equicontinuous if and only if it is strongly proximal.
\end{cor}
\begin{proof}
First assume that $(X,T)$ is mean equicontinuous. Then by Theorem \ref{thm:mean-eq-BP}
$BP(X,T)=X\times X$ which implies that $(X,T)$ is strongly proximal by \cite{LT13}.
Now assume that $(X,T)$ is strongly proximal. It is shown in that \cite{LT13} that $(X,T)$ is strongly proximal
if and only if $(X\times X,T\times T)$ is uniquely ergodic.
Thus $(X,T)$ is mean equicontinuous by Theorem~\ref{thm:mean-equi-unifon-conv}.
\end{proof}

In \cite{S82} the author asked the following question: does every ergodic invariant measure on
a mean equicontinuous system have discrete spectrum?
We show that this question has a positive answer. We note that the proof is inspired by~\cite[Theorem~4.4]{HLSY03}.

\begin{thm}\label{thm:mean-eq-discrete-spectrum}
Let $(X,T)$ be a dynamical system.
If $(X,T)$ is mean equicontinuous, then every ergodic invariant measure on $(X,T)$ has discrete spectrum
and hence the topological entropy of $(X,T)$ is zero.
\end{thm}
\begin{proof}
Let $\mu$ be an ergodic invariant measure on $(X,T)$.
Without loss of generality, assume that $\supp(\mu)=X$ and
then $(X,T)$ is uniquely ergodic by Corollary~\ref{cor:mean-eq-erg}.
Let $(Y,S)$ be the maximal equicontinuous factor of $(X,T)$ and $\pi:(X,T)\to (Y,S)$ be the factor map.
Then $\pi$ is proximal by Theorem \ref{thm:mean-eq-BP}.

Let $\nu$ be the unique invariant measure on $(Y,S)$. We have $\pi(\mu)=\nu$.
We consider the disintegration of $\mu$ over $\nu$.
That is, for a.e. $y\in Y$ we have a measure $\mu_y$ on $X$ such that $\supp(\mu_y)\subset \pi^{-1}(y)$
and $\mu=\int_{y\in Y}\mu_y d\nu$.
Let $W=\{(x,y)\in X\times X: \pi(x)=\pi(y)\}$.
As $\supp(\mu_y)\subset \pi^{-1}(y)$, a.e. $y\in Y$
we have $\supp(\mu_y\times\mu_y)\subset \pi^{-1}(y)\times \pi^{-1}(y)\subset W$, a.e. $y\in Y$.
Let $\mu\times_Y\mu=\int_{y\in Y}\mu_y\times\mu_yd\nu$.
Then $\mu\times_Y\mu$ is an invariant measure on $(X\times X,T\times T)$.
Moreover,
\[\mu\times_Y\mu(W)=\int_{y\in Y} \mu_y\times\mu_y(W)d\nu=1,\]
then $\supp(\mu\times_Y\mu)\subset W$.
By Theorem~\ref{thm:mean-eq-BP} every point in $W$ is Banach proximal. Thus we have $\supp(\mu\times_Y\mu)\subset \Delta_X$ (otherwise, there
exists a point $z$ in $W\setminus \Delta_X$ which returns to any neighborhood $U$ of $z$ with positive Banach density),
and $\mu\times_Y\mu(\Delta_X)=1$. Since
\[\mu\times_Y\mu(\Delta_X)=\int_{y\in Y}\mu_y\times \mu_y(\Delta_X)d\nu=1,\]
we have $\mu_y\times\mu_y(\Delta_X)=1$ a.e. $y\in Y$. By Fubini's
theorem we get that for a.e. $y\in Y$, $\mu_y$ is a combination of
countably many  atomic measures. Using the fact
$\mu_y\times\mu_y(\Delta_X)=1$, we conclude that for a.e. $y\in Y$,
there exists a point $c_y\in \pi^{-1}(y)$ such that
$\mu_y=\delta_{c_y}$.

Let $Z_0$ be the collection of $y\in Y$ such that $\mu_y$ is not equal to $\delta_x$ for
any $x\in X$. Then $\nu(\cup_{i\in \Z}S^{-i}Z_0)=0$. Let $Y_0=Y\setminus \cup_{i\in \Z}S^{-i}Z_0$
and $X_0=\{c_y:y\in Y_0\}$.
Then $\nu(Y_0)=1$. Now we show that $X_0$ is a measurable set. In
fact, the map $y\mapsto\mu_y$ from $Y_0$ to $M(X)$ is measurable and
$x\mapsto \delta_x$ is an embedding. Since $Y_0$ is a measurable set and
maps are 1-1, it follows from Souslin's theorem that $X_0$ is a
measurable set, and it is clear that $\mu(X_0)=\mu(\pi^{-1}Y_0)=\nu(Y_0)=1$. By the same
argument, $\pi: X_0\to Y_0$ is an isomorphism. This
shows that $\mu$ is isomorphic to $\nu$, and thus has discrete
spectrum.
\end{proof}

\begin{rem} We have the following remarks.
\begin{enumerate}
\item If $(X\times X,T\times T)$ is uniquely ergodic,
then by Theorems~\ref{thm:unique-ergodic} and \ref{thm:mean-equi-unifon-conv} $(X,T)$ is mean equicontinuous.
It is shown in \cite{LT13} that a dynamical system is strongly proximal
if and only if $(X\times X,T\times T)$ is uniquely ergodic.
There exist some strongly mixing systems which are strongly proximal, and thus they are also mean equicontinuous.
Since there exists a uniquely ergodic minimal system with positive entropy \cite{Ha67},
unique ergodicity does not imply mean equicontinuity by Theorem~\ref{thm:mean-eq-discrete-spectrum}.

\item
The authors in~\cite{HLY} have studied equicontinuous subsets and equicontinuity with respect to an invariant measure.
Let $(X,T)$ be a dynamical system and $K$ be a subset of $X$.
We say that $K$ is \emph{equicontinuous} if for any $\ep>0$ there is $\delta>0$
such that when $x,y\in K$ with $d(x,y)<\delta$, then $d(T^nx,T^ny)<\ep$ for all $n\geq 0$.
Let $\mu\in M(X,T)$. We say that $(X,T)$ is {\it $\mu$-equicontinuous} if for any $\tau>0$ there is
a $T$-equicontinuous compact subset $K$ of $X$ satisfying $\mu(K)>1-\tau$.
It is shown in~\cite[Corollary 5.6]{HLY} that if $(X,T)$ is $\mu$-equicontinuous then
$\mu$ has discrete spectrum. Using the same idea, Garc\'ia-Ramos
in a recent preprint \cite{Felipe} defines the notion of {\it $\mu$-mean equicontinuous},
and shows that $(X,T,\mu)$ is $\mu$-mean equicontinuous if and only if $\mu$ has discrete spectrum.

\item The notion of {\it tightness} was defined by Furstenberg, see \cite{OW}. Since tightness is an
isomorphic invariant \cite[Proposition 3]{OW}, it is easy to see that
if $(X,T)$ is mean equicontinuous and $\mu\in M(X,T)$ is ergodic, then $(X,T,\mu)$ is tight.
\end{enumerate}
\end{rem}

In~\cite{EG60}, Ellis and Gottschalk proved the existence of a maximal equicontinuous factor in dynamical system.
It is easy to see that mean equicontinuity satisfies the Remarks 6 and 7 in~\cite{EG60},
i.e. mean equicontinuity is preserved under subsystems and products, so we have the following result.
\begin{thm}\label{thm:max-mean-equi-factor}
Each dynamical system admits a maximal mean equicontinuous factor.
\end{thm}

To end the section we show that mean continuity is preserved by factor maps.

\begin{thm}\label{thm:mean-equi-factor}
Let $\pi:(X,T)\to (Y,S)$ be a factor map between two dynamical systems.
If $(X,T)$ is mean equicontinuous, then so is $(Y,S)$.
\end{thm}

\begin{proof}
Assume the contrary that $(Y,S)$ is not mean equicontinuous.
Then there is $\delta>0$ and two sequences $\{y_k\},\{z_k\}$ of points in $Y$
 with $d(y_k,z_k)<\frac 1 k$ such that for each $k\geq 1$,
\[\limsup_{n\to\infty} \frac 1n\sum_{i=0}^{n-1}d(S^i(y_k),S^i(z_k))\geq 2\delta.\]
By the compactness of $Y$, we may assume that $\lim_{k\to\infty}y_k=\lim_{k\to\infty}z_k=y\in Y$.
For each $k\geq 1$,
\[\limsup\limits_{n\to\infty} \frac 1n\sum\limits_{i=0}^{n-1}d(S^i(y_k),S^i(y))+
\limsup\limits_{n\to\infty} \frac 1n\sum\limits_{i=0}^{n-1}d(S^i(z_k),S^i(y))\geq
\limsup\limits_{n\to\infty} \frac 1n\sum\limits_{i=0}^{n-1}d(S^i(y_k),S^i(z_k)). \]
Then either $\limsup\limits_{n\to\infty} \frac 1n\sum\limits_{i=0}^{n-1}d(S^i(y_k),S^i(y))\geq \delta$
or  $\limsup\limits_{n\to\infty} \frac 1n\sum\limits_{i=0}^{n-1}d(S^i(z_k),S^i(y))\geq \delta$.
Without loss of generality, assume that
\[\limsup\limits_{n\to\infty} \frac 1n\sum\limits_{i=0}^{n-1}d(S^i(y_k),S^i(y))\geq \delta\]
 holds for all $k\geq 1$.
Choose a sequence $\{x_k\}$ in $X$ with $\pi(x_k)=y_k$.
By the compactness of $X$, we can assume that $\lim_{k\to\infty} x_k=x$. Then $\pi(x)=y$.

Choose an $\eta>0$ such that $\eta<\delta/(1+\diam(Y))$.
By the continuity of $\pi$, there exists $\theta\in (0,\eta)$ such that if $d(u,v)<\theta$ then $d(\pi(u),\pi(v))<\eta$.
Since $(X,T)$ is mean equicontinuous,
there is $\ep>0$ such that
\[\limsup_{n\to\infty} \frac 1n\sum\limits_{i=0}^{n-1}d(T^iw,T^ix)<\theta^2\]
for all $w\in B(x,\ep)$.
Choose $x_j\in B(x,\ep)$.
Let $E=\{k\in\Z: d(T^kx_j,T^kx)\geq \theta\}$. One has
\begin{align*}
  \theta^2>\limsup_{n\to\infty} \frac{1}{n}\sum_{i=0}^{n-1}d(T^ix_j,T^ix)\geq
  \limsup_{n\to\infty} \frac{1}{n} (\theta\#([0,n-1]\cap E)=\theta\overline{D}(E),
    \end{align*}
and then $\overline{D}(E)<\theta$.
Let $F=\{k\in\Z: d(T^ky_j,T^ky)\geq \eta\}$.
By the choice of $\theta$, we have $F\subset E$. Then $\overline{D}(F)<\theta$ and
\begin{align*}
  \limsup_{n\to\infty} \frac{1}{n}\sum_{i=0}^{n-1}d(S^i(y_j),S^i(y))&\leq
  \limsup_{n\to\infty} \frac{1}{n} (\diam(Y)\#([0,n-1]\cap F)+\eta n)\\
&\leq \diam(Y)\overline{D}(F)+\eta\leq \diam(Y)\theta+\eta \\
&\leq (\diam(Y)+1)\eta<\delta,
   \end{align*}
which is a contradiction.
Thus $(Y,S)$ is mean equicontinuous.
\end{proof}

\section{Almost mean equicontinuity}
In this section, we study the localization of mean equicontinuity.
Let $(X,T)$ be a dynamical system. A point $x\in X$ is called \emph{mean equicontinuous}
if for every $\ep>0$, there is $\delta>0$ such that for every $y\in B(x,\delta)$,
\[\limsup_{n\to\infty}\frac 1n \sum_{i=0}^{n-1}d(T^ix,T^iy)<\ep.\]
By the compactness of $X$, $(X,T)$ is mean equicontinuous if and only if every point in $X$ is mean equicontinuous.
A transitive system $(X,T)$ is called \emph{almost mean equicontinuous}
if there is at least one mean equicontinuous point.
We show that for a transitive system $(X,T)$, the set of mean equicontinuous points is either empty or residual.
If in addition $(X,T)$ is almost mean equicontinuous, then every transitive point is mean equicontinuous.
While almost equicontinuous systems must have zero topological entropy,
we construct many almost mean equicontinuous systems which have positive topological entropy.

Similarly to Theorem~\ref{thm:mean-equi-unifon-conv},
we have the following characterization of mean equicontinuous points.

\begin{thm}\label{thm:mean-equi-point}
Let $(X,T)$ be a dynamical system and $x\in X$. Then the following conditions are equivalent:
\begin{enumerate}
  \item $x$ is a mean equicontinuous point in $(X,T)$;
  \item for each $f\in C(X\times X)$, the sequence $\{f_n\}_{n=1}^\infty$ is equicontinuous at $(x,x)$;
  \item for each $f\in C(X\times X)$, the sequence $\{f_n\}_{n=1}^\infty$ is convergent  at $(x,x)$;
 \item for each $f\in C(X\times X)$, $\bar{f}$ is continuous at $(x,x)$, where $\bar{f}(y)=\limsup\limits_{n\to\infty} f_n(y)$.
\end{enumerate}
\end{thm}

Let $\mathcal{E}$  denote the set of all mean equicontinuous points.
For every $\ep>0$, let
\[\mathcal{E}_{\ep}=\Bigl\{x\in X: \exists \delta>0,\forall y,z\in B(x,\delta),\limsup_{n\to\infty}\frac{1}{n}\sum_{i=0}^{n-1}d(T^iy,T^iz)<\ep\Bigr\}.\]

\begin{prop}\label{pro-weiss}
Let $(X,T)$ be a dynamical system and $\ep>0$.
Then $\mathcal{E}_{\ep}$ is open and inversely invariant, that is $T^{-1}\mathcal{E}_{\ep}\subset \mathcal{E}_{\ep}$.
Moreover, $\mathcal{E}=\bigcap_{m=1}^\infty \mathcal{E}_{\frac{1}{m}}$ is a $G_\delta$ subset of $X$.
\end{prop}
\begin{proof}
Let $x\in\mathcal{E}_\ep$.
Choose $\delta>0$ satisfying the condition from the definition of $\mathcal{E}_\ep$
for $x$. Fix $y\in B(x,\frac{\delta}{2})$.
If $z,w\in B(y,\frac{\delta}{2})$,
then $z,w\in B(x,\delta)$, so
\[\limsup_{n\to\infty}\frac{1}{n}\sum_{i=0}^{n-1}d(T^iz,T^iw)<\ep.\]
This shows that $B(x,\frac{\delta}{2})\subset \mathcal{E}_{\ep}$ and thus $\mathcal{E}_{\ep}$ is open.

Let $x\in X$ with $Tx\in\mathcal{E}_\ep$.
Choose $\delta>0$ satisfying the condition from the definition of $\mathcal{E}_\ep$ for $Tx$.
By the continuity of $T$, there exists $\eta>0$ such that
$d(Ty,Tz)<\delta$ for any $y,z\in B(x,\eta)$.
If $y,z\in B(x,\eta)$, then $Ty,Tz\in B(Tx,\delta)$. Thus
\[
\limsup_{n\to\infty}\frac{1}{n}\sum_{i=0}^{n-1}d(T^iy,T^iz)=
\limsup_{n\to\infty}\frac{1}{n}\sum_{i=0}^{n-1}d(T^i(Ty),T^i(Tz))<\ep,
\]
This implies $x\in \mathcal{E}_\ep$.

If $x\in X$ belongs to all $\mathcal{E}_{\frac{1}{m}}$, then clearly $x\in\mathcal{E}$.
Conversely, if $x\in \mathcal{E}$ and $m>0$, then there exists $\delta>0$ such that
\[\limsup_{n\to\infty}\frac{1}{n}\sum_{i=0}^{n-1}d(T^ix,T^iy)<\frac{1}{2m}\]
for all $y \in B(x,\delta)$.
If $y,z\in B(x,\delta)$,
\begin{align*}
  \limsup_{n\to\infty}\frac{1}{n}\sum_{i=0}^{n-1}d(T^iy,T^iz)&\leq
\limsup_{n\to\infty}\frac{1}{n}\sum_{i=0}^{n-1}(d(T^ix,T^iy)+d(T^ix,T^iz))\\
&\leq \limsup_{n\to\infty}\frac{1}{n}\sum_{i=0}^{n-1}d(T^ix,T^iy)+
\limsup_{n\to\infty}\frac{1}{n}\sum_{i=0}^{n-1}d(T^ix,T^iz))\\
&<\frac{1}{2m}+\frac{1}{2m}=\frac{1}{m}.
\end{align*}
Thus $x\in \mathcal{E}_{\frac{1}{m}}$. This ends the proof.
\end{proof}

\begin{prop}\label{prop:mean-equi-point}
Let $(X,T)$ be a transitive system.
\begin{enumerate}
\item The set of mean equicontinuous points is either empty or residual.
If in addition $(X,T)$ is almost mean equicontinuous,
then every transitive point is mean equicontinuous.
\item If $(X,T)$ is minimal and almost mean equicontinuous, then it is mean equicontinuous.
\end{enumerate}
\end{prop}
\begin{proof}
By the transitivity of $(X,T)$, every $\mathcal{E}_{\ep}$ is either empty or dense, since
$\mathcal{E}_{\ep}$ is open and inversely invariant.
Then $\mathcal{E}$ is either empty or residual by the Baire Category Theorem.

If $\mathcal{E}$ is residual, then every $\mathcal{E}_{\ep}$ is open and dense.
Let $x\in X$ be a transitive point and $\ep>0$.
Then there exists some $k\in\Z$ such that $T^kx\in \mathcal{E}_{\ep}$,
and since $\mathcal{E}_{\ep}$ is inversely invariant, $x\in \mathcal{E}_{\ep}$.
Thus $x\in \mathcal{E}$.
\end{proof}

\begin{rem} We have the following remarks.
\begin{enumerate}
\item It should be noticed that the set of mean equicontinuous points may not coincide with the set of transitive points.
For example, a weakly mixing strongly proximal system is mean equicontinuous,
but the set of non-transitive points is dense in the space.

\item Every almost equicontinuous system is almost mean equicontinuous.
There exists some almost equicontinuous systems which have more than
one fixed point (see~\cite{HY}) and are thus not uniquely ergodic. So those systems are
almost equicontinuous but not mean equicontinuous.

\item It is shown in~\cite{GW93} that factors of almost equicontinuous systems may not be almost equicontinuous.
In fact,  factors of almost equicontinuous systems may not even be almost mean equicontinuous.
Let $(Y,S)$ be a uniformly rigid weakly mixing minimal system.
By Corollary~\ref{cor:mean-eq-wm} $(Y,S)$ is not almost mean equicontinuous.
By \cite[Proposition~1.5]{GW93} there exists an extension $(X,T)$ of $(Y,S)$ which is  almost equicontinuous.
\end{enumerate}
\end{rem}

Let $\pi:X\to Y$ be a map.
The map $\pi$ is called \emph{open} if for every non-empty open subset $U$ of $X$, $\pi(U)$ is open in $Y$,
and \emph{semi-open} if $\pi(U)$ has non-empty interior in $Y$.
We say that $\pi$ is \emph{open at a point $x\in X$} for for every neighborhood $U$ of $x$,
$\pi(U)$ is a neighborhood of $\pi(x)$.

\begin{lem}\label{lem:mean-equi-point-open}
Let $\pi:(X,T)\to (Y,S)$ be a factor map between two dynamical systems..
Let $x\in X$ be a mean equicontinuous point and suppose that $\pi$ is open at $x$.
Then $y=\pi(x)$ is also a mean equicontinuous point.
\end{lem}
\begin{proof}
If $y$ is not a mean equicontinuous point,
then there exists $\delta>0$ and a sequence $y_k\to y$ such that for any $k$,
\[\limsup_{n\to\infty} \frac 1n \sum_{i=0}^{n-1}d(S^iy_k,S^iy)>\delta.\]
Then there exists $t>0$ such that for any $k\geq 1$,
the upper density of $\{n\in\Z: d(S^iy_k,S^i y)>t\}$ is greater than $t$.

By the openness of $\pi$ at $x$, there is a sequence $x_k\to x$ with $\pi(x_k)=y_k$.
Since $x$ is mean equicontinuous, for every $\ep>0$, for large enough $k$, one has
\[\limsup_{n\to\infty} \frac 1n \sum_{i=0}^{n-1}d(T^ix_k,T^ix)<\ep.\]
By the uniform continuity of $\pi$, there exists $s>0$ such that if $d(u,v)<s$, then $d(\pi(u),\pi(v))<t/2$.
For this $s$, there exists $k\geq 1$ such that
the upper density of $\{n\in\Z: d(T^ix_k,T^ix)>s\}$ is less than $t$,
then the upper density of $\{n\in\Z: d(S^iy_k,S^i y)>t\}$ is less than $t$.
This is a contradiction. Thus $y$ is a mean equicontinuous point.
\end{proof}

\begin{thm}\label{thm:almost-semi-open}
Let $\pi:(X,T)\to(Y,S)$ be a factor map between transitive systems.
Suppose that $(X,T)$ is almost mean equicontinuous and $\pi$ is semi-open.
Then $(Y,S)$ is also almost mean equicontinuous.
\end{thm}
\begin{proof}
Since $\pi$ is semi-open,  by \cite[Lemma~2.1]{G07}
the set $\{x\in X: \pi\text{ is open at }x\}$ is residual in $X$.
Pick a transitive point $x\in X$ such that $\pi$ is open at $x$.
Note that  $\pi(x)$ is also a transitive point in $Y$.
Since $(X,T)$ is almost mean equicontinuous, $x$ is a mean equicontinuous point.
By Lemma~\ref{lem:mean-equi-point-open}, $\pi(x)$ is a mean equicontinuous point,
which implies that $(Y,S)$ is  almost mean equicontinuous.
\end{proof}

Recall that an almost equicontinuous system is uniformly rigid and thus has zero topological entropy.
The following Theorem \ref{p-entropy} shows that an almost mean equicontinuous system
behaves quite differently.

To start with we need some preparation. Let $\Sigma_2=\{0,1\}^\N$
with the product topology. A metric inducing the topology is given by $d(x,y)=0$ if
$x=y$, and $d(x,y)=\frac{1}{i}$ if $x\not=y$ and $i=\min\{i:x_i\not= y_i\}$
when $x=x_1x_2\ldots$ and $y=y_1y_2\ldots$.
For $n\in \N$, we call $A\in \{0,1\}^n$ a {\it finite block with length $n$}
and denote $|A|=n$.
For two blocks $A=x_1\ldots x_n$ and $B=y_1\ldots y_m$ define $AB=x_1\ldots x_ny_1\ldots y_m$.
For a block $A$, let $[A]$ be the collection of $x\in\Sigma_2$ starting from $A$.
For $x\in \Sigma_2$ and $i<j$, $x_{[i,j]}$ stands for
the finite block $x_ix_{i+1}\ldots x_j$.

The \emph{shift map} $\sigma:\Sigma_2\to\Sigma_2$ is defined by
the condition that $\sigma(x)_n=x_{n+1}$ for $n\in\N$.
It is clear that $\sigma$ is a continuous surjection.
The dynamical system $(\Sigma_2,\sigma)$ is called the \emph{full shift}.
If $X$ is non-empty, closed, and $\sigma$-invariant (i.e. $\sigma(X)\subset X$),
then the dynamical system $(X, \sigma)$ is called a \emph{subshift}.

\begin{thm}\label{p-entropy}
In the full shift $(\Sigma_2,\sigma)$, every minimal subshift $(Y,\sigma)$ is contained
in an almost mean equicontinuous subshift $(X,\sigma)$.
\end{thm}
\begin{proof}

If $Y=\{0^\infty\}$, the result is obvious. Now assume that $Y\neq \{0^\infty\}$.
Choose a point $y=y_1y_2\ldots\in Y$ starting with $1$.
For every $n\geq 1$, set $B_n=y_1\ldots y_n$.

Set $A_1=11, A_2=A_10^{k_1}B_10^{k_1}A_1$ and
\[A_{n+1}=A_n0^{k_n}B_n0^{k_n}A_n\]
for $n\geq 2$, where $k_{n}\gg3|A_n|$.
We will require the sequence $\{k_n\}$ to satisfy  some proper properties later.

Let $x=\lim_{n\to\infty}A_n0^\infty$ and $X=\overline{Orb(x,\sigma)}$.
It is clear that $(Y,\sigma)$ is a subsystem of $(X,\sigma)$.
We are going to show that $(X,\sigma)$ is almost mean equicontinuous.

For a finite block $A$ over $\{0,1\}$,
we denote by $\mathfrak{o}(A)$ the function which counts the number of ones in $A$.
Since $y$ is minimal, by Lemma~\ref{lem:rec-ip} the word $1$ appears in $y$ syndetically.
Then there is $N\in\N$ such that $A_N$ does not appear in $y$.
For any $n\geq N$, we can express $x$ as
\[x=\underline{A_n0^{k_n}B_n0^{k_n}}\ \underline{A_n0^{k_{n+1}}B_{n+1}0^{k_{n+1}}}\dotsb
\underline{A_n0^{k_m}B_{m}0^{k_m}}\ \underline{A_n0^{k_q}B_{q}0^{k_q}}\dotsb.\]

\begin{claim} \label{claim-1}
For any $m\geq n$ and $1\leq i\leq |A_n|+2k_m+m$, we have
\[ \mathfrak{o}(A_n0^{k_m}B_{m}0^{k_m}{}_{[1,i]})\leq \max\Bigl\{1, \frac{i}{|A_n|+2k_n+n}\Bigr\}\cdot(|A_n|+n).\]
\end{claim}
\begin{proof}
It is clear that we only need to prove the case $m>n$.
The result is obvious for $i\in[1,|A_n|+k_m]$.
If we require the sequence $\{k_n\}$ to satisfy
\[\frac{k_q}{|A_p|+2k_p+p}>\frac{q}{p}, \qquad \forall 1\leq p<q,\]
then $k_m>2k_n$ and $m<\frac{k_m \cdot n}{|A_n|+2k_n+n}.$
So for $i>|A_n|+k_m$, we have
\begin{align*}
  \mathfrak{o}(A_n0^{k_m}B_{m}0^{k_m}{}_{[1,i]})& \leq |A_n|+m  \leq |A_n|+\frac{k_m \cdot n}{|A_n|+2k_n+n}\\
  &\le \frac{(|A_n|+k_m)(|A_n|+n)}{ |A_n|+2k_n+n}\\
  &\leq \frac{i}{|A_n|+2k_n+n}(|A_n|+n). \qedhere
\end{align*}
\end{proof}

\begin{claim}\label{claim-2}
If $j-i>|A_n|$ and $x_{[i,i+|A_n|-1]}=A_n$, then we have
\[\mathfrak{o}(x_{[i,j-1]})\leq \left(\frac{j-i}{|A_n|+2k_n+n}+1\right)(|A_n|+n). \]
\end{claim}
\begin{proof}
Since $A_N$ does not appear in $y$ and $x_{[i,i+|A_n|-1]}$ is equal  to $A_n$,
$x_{[i,j-1]}$ can be expressed as
\[x_{[i,j-1]}=\underline{A_n0^{k_{i_1}}B_{i_1}0^{k_{i_1}}}\ \underline{A_n0^{k_{i_2}}B_{i_2}0^{k_{i_2}}}\dotsb
\underline{A_n0^{k_{i_m}}B_{i_m}0^{k_{i_m}}}\ \underline{A_n0^{k_{q}}B_{q}0^{k_{q}}[1,p]}\]
for $i_l\geq n$, $l=1,\ldots, m$  and for some $p<|A_n|+2k_{q}+q$.

Then by Claim~\ref{claim-1} we have
\[\mathfrak{o}(A_n0^{k_{i_l}}B_{i_l}0^{k_{i_l}})\le \frac{|A_n0^{k_{i_l}}B_{i_l}0^{k_{i_l}}|}{|A_n|+2k_n+n}(|A_n|+n)\]
for $l=1,\ldots, m$ and
\[\mathfrak{o}(A_n0^{k_{q}}B_{q}0^{k_{q}}{}_{[1,p]})\le \frac{p}{|A_n|+2k_n+n}(|A_n|+n)+|A_n|+n.\]

Thus
\begin{align*}
  \mathfrak{o}(x_{[i,j-1]})&\leq  \mathfrak{o}(A_n0^{k_{i_1}}B_{i_1}0^{k_{i_1}})
  +\dotsb+\mathfrak{o}(A_n0^{k_{i_m}}B_{i_m}0^{k_{i_m}})+
  \mathfrak{o}(A_n0^{k_q}B_{q}0^{k_q}{}_{[1,p]}) \\
  & \leq \left(\frac{j-i}{|A_n|+2k_n+n}+1\right)(|A_n|+n).  \qedhere
\end{align*}
\end{proof}

Fix a positive number $\ep$. There is $K\in\N$ such that for any $a,b\in \Sigma$ if $a[0,K-1]=b[0,K-1]$ then $d(a,b)<\ep/5$.
If we require the sequence $\{k_n\}$ to satisfy
\[\frac{|A_n|+n}{|A_n|+2k_n+n} \searrow 0  \text{ as } n\to\infty,\]
then we can choose a large enough integer $n$ such that
\begin{equation}\label{i-1}
\frac{2K(|A_n|+n)}{|A_n|+2k_n+n}<\ep/4.
\end{equation}

Fix a point $z\in [A_n]\cap X$. If there exists some $k\in\Z$ such that $T^kz=0^\infty$, then it is clear that
\[\limsup_{N\to\infty} \frac{1}{N}\sum_{i=0}^{N-1}d(T^iz,0^\infty)=0.\]

Now we assume that $T^kz\neq0^\infty$ for all $k\in\Z$.
For any $i\geq 1$, there is $m_i\in\N$ such that $z_{[0,i]}=x_{[m_i,m_i+i]}$.
Since $z$ starts with $A_n$, by Claim~\ref{claim-2} we have
\begin{equation}\label{i-2}
\mathfrak{o}(z_{[0,i-1]})\leq \left(\frac{i}{|A_n|+2k_n+n}+1\right)(|A_n|+n)\end{equation}
for all $i>|A_n|$.
By the choice of $K$, we have
\begin{align*}
\limsup_{N\to\infty} \frac{1}{N}\sum_{i=0}^{N-1}d(T^iz,0^\infty)&\leq
\frac{\ep}{5}\cdot\limsup_{N\to\infty} \frac{1}{N}{\#\{0\leq i\leq N-1\colon (T^iz)_{[0,K-1]}=0^K\}}\\
&\qquad +\limsup_{N\to\infty} \frac{1}{N}{\#\{0\leq i\leq N-1\colon (T^iz)_{[0,K-1]}\neq 0^K\}}\\
&< \ep/4 + \limsup_{N\to\infty} \frac{1}{N}{2K\cdot\mathfrak{o}(z_{[0,N-1]})} \le \ep/2,
\end{align*}
by (\ref{i-1}) and (\ref{i-2}). This implies that for any $z\in [A_n]\cap X$
\[\limsup_{N\to\infty} \frac{1}{N}\sum_{i=0}^{N-1}d(T^iz,T^ix)<\ep,\]
and thus $x$ is a mean equicontinuous point.
\end{proof}

Since it is well-known that there are many minimal subshifts of $(\Sigma_2,\sigma)$
with positive topological entropy,  an immediate corollary of Theorem~\ref{p-entropy} is the following result.
\begin{cor}\label{cor:almost-mean-equi-entropy}
There exist many almost mean equicontinuous systems which have positive topological entropy.
\end{cor}

\section{Mean sensitivity}
In this section we study the opposite side of mean equicontinuity---mean sensitivity.
It turns out that if a dynamical system $(X,T)$ is minimal then
$(X,T)$ is either mean equicontinuous or mean sensitive,
and if $(X,T)$ is transitive then $(X,T)$ is either almost mean equicontinuous or mean sensitive.

A dynamical system $(X,T)$ is \emph{mean sensitive}
if there exists $\delta>0$ such that for every $x\in X$ and every $\ep>0$
there is $y\in B(x,\ep)$ satisfying
\[\limsup_{n\to\infty}\frac{1}{n}\sum_{i=0}^{n-1}d(T^ix,T^iy)>\delta.\]

\begin{prop}\label{prop:mean-sensitive-equi}
Let $(X,T)$ be a dynamical system. Then the following conditions are equivalent:
\begin{enumerate}
  \item\label{enum:mean-sensitive} $(X,T)$ is mean sensitive;
  \item\label{enum:mean-sensitive-open} there exists $\delta>0$ such that
for every non-empty open subset $U$ of $X$
there are $x,y\in U$ satisfying
\[\limsup_{n\to\infty}\frac{1}{n}\sum_{i=0}^{n-1}d(T^ix,T^iy)>\delta;\]
\item\label{enum:x-cell} there exists $\eta>0$ such that
for every $x\in X$, the set
\[D_\eta(x)=\bigl\{y\in X\colon\limsup_{n\to\infty}\frac{1}{n}\sum_{i=0}^{n-1}d(T^ix,T^iy)\geq \eta\bigr\}\]
is a dense $G_\delta$ subset of $X$;
 \item\label{enum:XxX} there exists $\eta>0$ such that
\[D_\eta=\bigl\{(x,y)\in X\times X\colon\limsup_{n\to\infty}\frac{1}{n}\sum_{i=0}^{n-1}d(T^ix,T^iy)\geq \eta\bigr\}\]
is a dense $G_\delta$ subset of $X\times X$.
\end{enumerate}
\end{prop}
\begin{proof}
\eqref{enum:mean-sensitive}$\Rightarrow$\eqref{enum:XxX}
Since $(X,T)$ is mean sensitive, there exists $\delta>0$ such that for every $x\in X$ and every $\ep>0$
there is $y\in B(x,\ep)$ satisfying
\[\limsup_{n\to\infty}\frac{1}{n}\sum_{i=0}^{n-1}d(T^ix,T^iy)>\delta.\]
Let $\eta=\frac{\delta}{2}$. It is not hard to check that
\[D_\eta=\bigcap_{m=1}^{\infty}\bigcap_{\ell=1}^{\infty}
\Bigl(\bigcup_{n\ge \ell} \bigl\{(x,y)\in X\times X:
\frac{1}{n}\sum_{i=0}^{n-1}d(T^ix,T^iy)>\eta-\tfrac{1}{m}\bigr\}\Bigr).\]
Then $D_\eta$ is a $G_\delta$ subset of $X\times X$.
If $D_\eta$ is not dense in $X\times X$, then there exist two non-empty open subsets $U$ and $V$ of $X$
such that
\[U\times V\subset \bigl\{(x,y)\in X\times X: \limsup_{n\to\infty}\frac{1}{n}\sum_{i=0}^{n-1}d(T^ix,T^iy)<\eta\bigr\}.\]
Pick $x\in U$, $z\in V$  and $\ep>0$ such that $B(x,\ep)\subset U$.
Then for every $y\in B(x,\ep)$, one has
\begin{align*}
\limsup_{n\to\infty}\frac{1}{n}\sum_{i=0}^{n-1}d(T^ix,T^iy)
&\leq\limsup_{n\to\infty}\frac{1}{n}\sum_{i=0}^{n-1}(d(T^ix,T^iz)+d(T^iy,T^iz))\\
&\leq\limsup_{n\to\infty}\frac{1}{n}\sum_{i=0}^{n-1}d(T^ix,T^iz)+
\limsup_{n\to\infty}\frac{1}{n}\sum_{i=0}^{n-1}d(T^iy,T^iz)\\
&<\eta+\eta=\delta,
\end{align*}
which is a contradiction. Thus $D_\eta$ is a dense $G_\delta$ subset of $X\times X$.

\eqref{enum:XxX}$\Rightarrow$\eqref{enum:x-cell}
Assume that there exists $\eta>0$ such that $D_{2\eta}$ is a dense $G_\delta$ subset of $X\times X$.
Note that $D_\eta(x)$ is a $G_\delta$ subset of $X$. Thus it suffices to show that $D_\eta(x)$ is dense.
For every a non-empty open subset $U$ of $X$, there exist two points $y, z\in U$ with $(y,z)\in D_{2\eta}$,
that is
\[\limsup_{n\to\infty}\frac{1}{n}\sum_{i=0}^{n-1}d(T^iy,T^iz)>2\eta.\]
Then
\begin{align*}
2\eta< \limsup_{n\to\infty}\frac{1}{n}\sum_{i=0}^{n-1}d(T^iy,T^iz)
&\leq\limsup_{n\to\infty}\frac{1}{n}\sum_{i=0}^{n-1}(d(T^ix,T^iy)+d(T^ix,T^iz))\\
&\leq\limsup_{n\to\infty}\frac{1}{n}\sum_{i=0}^{n-1}d(T^ix,T^iy)+ \limsup_{n\to\infty}\frac{1}{n}\sum_{i=0}^{n-1}d(T^ix,T^iz).
\end{align*}
Thus, either
\[\limsup_{n\to\infty}\frac{1}{n}\sum_{i=0}^{n-1}d(T^ix,T^iy)>\eta\]
or
\[\limsup_{n\to\infty}\frac{1}{n}\sum_{i=0}^{n-1}d(T^ix,T^iz)>\eta,\]
that is either $y\in D_\eta(x)$ or $z\in D_\eta(x)$.

\eqref{enum:x-cell}$\Rightarrow$\eqref{enum:mean-sensitive-open} is obvious.

\eqref{enum:mean-sensitive-open}$\Rightarrow$\eqref{enum:mean-sensitive}
For any $x\in X$ and $\ep>0$, there are $y,z\in B(x,\ep)$
satisfying
\[\limsup_{n\to\infty}\frac{1}{n}\sum_{i=0}^{n-1}d(T^iy,T^iz)>\delta.\]
Then either
\[\limsup_{n\to\infty}\frac{1}{n}\sum_{i=0}^{n-1}d(T^ix,T^iy)>\frac{\delta}{2}\]
or
\[\limsup_{n\to\infty}\frac{1}{n}\sum_{i=0}^{n-1}d(T^ix,T^iz)>\frac{\delta}{2},\]
which implies that $(X,T)$ is mean sensitive.
\end{proof}

A point $x\in X$ is \emph{mean sensitive} if there exists $\delta>0$ such that for every $\ep>0$
there is $y\in B(x,\ep)$ satisfying
\[\limsup_{n\to\infty}\frac{1}{n}\sum_{i=0}^{n-1}d(T^ix,T^iy)>\delta.\]
It is clear that a point is either mean equicontinuous or mean sensitive.
The following example indicates that although in a dynamical system every point is mean sensitive,
the dynamical system may not be sensitive.
\begin{exmp}
Let $T\colon [0,1]\to [0,1]$ be the standard tent map, that is $T(x)=1-|1-2x|$.
Let
\[Y=\{(x,0)\colon x\in[0,1]\}\cup\bigcup_{k=1}^\infty \{-\tfrac{1}{k}\}\times[0,\tfrac{1}{k}],\]
be endowed with metric induced by the Euclidean metric.
For $x\in [0,1]$, put $S(x,0)=S(Tx,0)$, and for $k\geq 1$ and $x\in[0, \tfrac{1}{k}]$, put
$S(-\tfrac{1}{k},x)=(-\tfrac{1}{k},\tfrac{1}{k} T(kx))$.
It is not hard to verify that every point in $(Y,S)$ is mean sensitive, but $(Y,S)$ is not sensitive.
\end{exmp}

\begin{prop}\label{prop:mean-sensitive-point}
Let $(X,T)$ be a transitive system.
If there exists a transitive point which is mean sensitive, then $(X,T)$ is mean sensitive.
\end{prop}
\begin{proof}
Let $x\in X$ be a mean sensitive transitive point, that is
there exists $\delta>0$ such that for every $\ep>0$
there is $y\in B(x,\ep)$ satisfying
\[\limsup_{n\to\infty}\frac{1}{n}\sum_{i=0}^{n-1}d(T^ix,T^iy)>\delta.\]
Fix a non-empty open subset $U$ of $X$.
As $x$ is a transitive point, there exist $\ep>0$ and $k\in\Z$ such that
$T^k(B(x,\ep))\subset U$. There exists $y\in B(x,\ep)$ satisfying
\[\limsup_{n\to\infty}\frac{1}{n}\sum_{i=0}^{n-1}d(T^ix,T^iy)>\delta.\]
Let $u=T^kx$ and $v=T^ky$. Then $u,v\in U$ and
\begin{align*}
\limsup_{n\to\infty}\frac{1}{n}\sum_{i=0}^{n-1}d(T^iu,T^iv)&=
\limsup_{n\to\infty}\frac{1}{n}\sum_{i=0}^{n-1}d(T^i(T^kx),T^i(T^ky))\\
&=\limsup_{n\to\infty}\frac{1}{n}\sum_{i=0}^{n-1}d(T^ix,T^iy)
>\delta.
\end{align*}
Therefore $(X,T)$ is mean sensitive by Proposition~\ref{prop:mean-sensitive-equi}.
\end{proof}

Now we have the following dichotomy for transitive systems.
\begin{thm}\label{almostmeorms}
Let $(X,T)$ be a dynamical system.
If $(X,T)$ is transitive, then $(X,T)$ is either almost mean equicontinuous or mean sensitive.
\end{thm}
\begin{proof}
Let $x\in X$ be a transitive point.
If $x$ is mean sensitive, then $(X,T)$ is mean sensitive by Proposition~\ref{prop:mean-sensitive-point}.
If $x$ is not mean sensitive, then it is mean equicontinuous.
So $(X,T)$ is almost mean equicontinuous by Proposition~\ref{prop:mean-equi-point}.
\end{proof}

\begin{cor}\label{meorms}
Let $(X,T)$ be a minimal system.
Then $(X,T)$ is either mean equicontinuous or mean sensitive.
\end{cor}

\begin{rem}\label{rem:mean-sensitive-point}
In fact, from the proof of Theorem~\ref{thm:mean-equi-factor} we have that
if $\pi:(X,T)\to (Y,S)$ is a factor map and $y\in Y$ is a mean sensitive point,
then there exists a mean sensitive point $x\in \pi^{-1}(y)$.
\end{rem}

By Lemma~\ref{lem:mean-equi-point-open}, we have the following result.
\begin{lem}
Let  $\pi:(X,T)\to (Y,S)$ be a factor map.
Assume that $y\in Y$ is a mean sensitive point and $x\in \pi^{-1}(y)$.
If $\pi$ is open at $x$, then $x$ is also a mean sensitive point.
\end{lem}

Similarly to Theorem~\ref{thm:almost-semi-open}, we have

\begin{thm}
Let  $\pi:(X,T)\to (Y,S)$ be a factor map between transitive systems.
Assume that $(Y,S)$ is mean sensitive. If $\pi$ is semi-open,  then $(X,T)$ is mean sensitive.
\end{thm}

\begin{prop}
Let $\pi:(X,T)\to (Y,S)$ be a factor map between transitive systems.
If $(Y,S)$ is mean sensitive, then there is a mean sensitive subsystem $(Z,T)$ of $(X,T)$.
\end{prop}
\begin{proof}
Using Zorn's Lemma and the compactness of $X$, we can find a subsystem  $(Z,T)$ of $(X,T)$ such that
$\pi(Z)=Y$ and $Z$ is minimal with respect to this property.
Let $y\in Y$ be a transitive point which is also a mean sensitive point.
By Remark~\ref{rem:mean-sensitive-point} there is a mean sensitive point $x\in (Z,T)$ with $\pi(x)=y$.
Let $Z'=\omega(x,T)$. Then $(Z',T)$ is a subsystem of $(Z,T)$ with $\pi(Z')=Y$.
The minimality of $Z$ implies that $Z'=Z$.
Then $x$ is a transitive point of $(Z,T)$ and $(Z,T)$ is mean sensitive.
\end{proof}

\section{Banach mean equicontinuity and Almost Banach mean equicontinuity}
Globally speaking a mean equicontinuous system is `simple', since it is a Banach proximal extension of
an equicontinuous system and each of its ergodic measures has discrete spectrum.
Unfortunately, the local version does not behave so well, as Theorem~\ref{p-entropy} shows.
In this section we introduce the notion of Banach mean equicontinuity,
whose local version has the better behavior that we are looking for. 

Let $(X,T)$ be a dynamical system. We say that $(X,T)$ is \emph{Banach mean equicontinuous} ({\bf BME} for short)
if for every $\ep>0$,  there exists $\delta>0$ such that whenever $x,y\in X$ with $d(x,y)<\delta$,
\[\limsup_{M-N\to\infty}\frac{1}{M-N}\sum_{i=N}^{M-1}d(T^ix,T^iy)<\ep.\]

A point $x\in X$ is called \emph{{\bf BME}}
if for every $\ep>0$, there is $\delta>0$ such that for every $y\in B(x,\delta)$,
\[\limsup_{M-N\to\infty}\frac{1}{M-N}\sum_{i=N}^{M-1}d(T^ix,T^iy)<\ep.\]
We say that a transitive system $(X,T)$ is \emph{almost Banach mean equicontinuous} ({\bf ABME} for short)
if there is at least one {\bf BME} point.

By the compactness of $X$, $(X,T)$ is {\bf BME} if and only if every point in $X$ is {\bf BME}.
Following the proofs of Propositions~ \ref{pro-weiss} and ~\ref{prop:mean-equi-point}
we know that the set of {\bf BME} points is a $G_\delta$ set, and if $(X,T)$
is {\bf ABME}, then every transitive point is {\bf BME}.

A dynamical system $(X,T)$ is \emph{Banach mean sensitive} ({\bf BMS} for short)
if there exists $\delta>0$ such that for every $x\in X$ and every $\ep>0$
there is $y\in B(x,\ep)$ satisfying
\[\limsup_{M-N\to\infty}\frac{1}{M-N}\sum_{i=N}^{M-1}d(T^ix,T^iy)>\delta.\]

We list several properties whose proofs are similar to ones in the previous sections.

\begin{prop}\label{property-6}
Let $(X,T)$ be a dynamical system.
\begin{enumerate}
\item $(X,T)$ is {\bf BME} if and only if for every $\ep>0$,
there is a $\delta>0$ such that $d(x,y)<\delta$ implies $d(T^nx,T^ny)<\ep$ for all $n\in\Z$ except a set of
upper Banach density less than $\ep$.
\item $(X,T)$ is {\bf BMS} if and only if
  there is a $\delta>0$ such that
  for every $x\in X$ and every neighborhood $U$ of $x$, there is $y\in U$ such that
  $d(T^nx,T^ny)>\delta$ in a set of upper Banach density larger than $\delta$.
\item Every dynamical system has a maximal {\bf BME} factor.
\item Let $\pi:(X,T)\to (Y,S)$ be a factor map.
If $(X,T)$ is {\bf BME}, then so is $(Y,S)$.
\item If $(X,T)$ is a transitive system,
then $(X,T)$ is either {\bf ABME} or {\bf BMS}; and if $(X,T)$ is a minimal system, then $(X,T)$ is either  {\bf BME} or {\bf BMS}.
\item Let $\pi\colon(X,T)\to (Y,S)$ be a factor map between transitive systems. Suppose that $\pi$ is semi-open.
\begin{enumerate}
  \item If $(X,T)$ is {\bf ABME}, then so is $(Y,S)$.
  \item If $(Y,S)$ is {\bf BMS}, then so is $(X,T)$.
\end{enumerate}
\end{enumerate}
\end{prop}

The main result in this section is:
\begin{thm}\label{thm:pe-bms}
Let $(X,T)$ be a transitive system.
If the topological entropy of $(X,T)$ is positive, then $(X,T)$ is Banach mean sensitive.
\end{thm}

To prove the above theorem we need some preparation. The following result is implicit in~\cite{Gi36},
see also \cite[Proposition 5.8]{HLY} for this version.
\begin{prop}
Let $(X,\mathcal{B},\mu)$ be a probability space, and
$\{E_i\}_{i=1}^\infty$ be a sequence of measurable sets with $\mu(E_i)\geq a>0$
for some constant $a$ and any $i\in\N$.
Then for any $k\geq 1$ and $\ep>0$, there is $N=N(a,k,\ep)$ such that
for any tuple $\{s_1<s_2<\dotsb<s_n\}$ with $n\geq N$ there exist $1\leq t_1<t_2<\dotsb<t_k\leq n$
with
\[\mu(E_{s_{t_1}}\cap E_{s_{t_2}}\cap\dotsb\cap E_{s_{t_k}})\geq a^k-\ep.\]
\end{prop}

By the well-known Furstenberg corresponding principle~\cite{F81}, we have

\begin{prop}\label{prop:BD-delta}
Let $S$ be a subset of $\Z$ with $BD^*(S)>0$.
Then for any $k\geq 1$ and $\ep>0$, there is $N=N(a,k,\ep)$ such that
for any tuple $\{s_1<s_2<\dotsb<s_n\}$ with $n\geq N$ there exist $1\leq t_1<t_2<\dotsb<t_k\leq n$
with
\[BD^*((S-s_{t_1})\cap (S-s_{t_2})\cap\dotsb\cap (S-s_{t_k}))\geq (BD^*(S))^k-\ep.\]
\end{prop}

\begin{lem} \label{lem:pubd-ip}
If $S$ has positive upper Banach density and $W$ is an IP-set,
then there are $q,q_1,q_2\in S$ and
$l_1,l_2\in W$ with $l_i=q_i-q$ for each $i=1,2$ and
\[BD^*((S-l_1)\cap(S-l_2))\geq \frac 14(BD^*(S))^4.\]
\end{lem}
\begin{proof}
Let $W$ be the IP-set generated by $\{p_i\}_{i=1}^\infty$.
Without loss of generality, assume that $\sum_{i=1}^n p_i<p_{n+1}$.
For the sequence $\{\sum_{i=1}^n p_i\}_{i=1}^\infty$,
by Proposition~\ref{prop:BD-delta}, there exist $n_1<n_2$ such that
\[BD^*((S-\sum_{i=1}^{n_1}p_i)\cap (S-\sum_{i=1}^{n_2}p_i))\geq \frac 12BD^*(S)^2.\]
Let $r_1=\sum_{i=n_1}^{n_2} p_i$ and $S_1=S\cap (S-r_1)$.
Then $r_1\in W$ and $BD^*(S_1)\geq \frac 12BD^*(S)^2$.
Let $W_1$ be the IP-set generated by $\{p_i\}_{i=n_2+1}^\infty$.
Again by Proposition~\ref{prop:BD-delta},
there is $r_2\in W$ such that $S_2=S_1\cap (S_1-r_2)$ satisfying
$BD^*(S_2)\geq \frac 12BD^*(S_1)^2$.
Let $l_1=r_1$ and $l_2=r_1+r_2$. Choose $q\in S_2$.
Then $q_1=q+r_1\in (S_1+r_1)\subset S$,
\[q_2=q+r_1+r_2\in (S_2+r_2)+r_1\subset (S_1+r_1)\subset S\]
and
\[BD^*((S-l_1)\cap(S-l_2))\geq BD^*(S_2) \geq \frac 12 BD^*(S_1)^2 \geq  \frac 14(BD^*(S))^4. \qedhere\]
\end{proof}

The notion of entropy pair was introduced by Blanchard in \cite{B2}.
It is known that if the topological entropy of a dynamical system is
positive then there exist some entropy pairs.
We have the following characterization of entropy pairs.

\begin{prop}[\cite{HY06,KL07}] \label{thm:entropy-pair}
Let $(X,T)$ be a dynamical system.
Then $(x_1,x_2)\in X\times X$ is an entropy pair if and only if for each neighborhood $U_i$ of $x_i$ for $i=1,2$,
there is an independence set of positive density for $(U_1,U_2)$, i.e.
there is a subset $J\subset \Z$ with positive density such that for any non-empty finite subset $I\subset J$,
we have
\[\bigcap_{i\in I}T^{-i}U_{s(i)}\neq\emptyset\]
for any $s\in \{1,2\}^I$.
\end{prop}

Now we are ready to give a proof of Theorem~\ref{thm:pe-bms}.
We remark that some idea of the proof is
inspired by \cite{YZ}.

\begin{proof}[Proof of Theorem~\ref{thm:pe-bms}]
Since the topological entropy of $(X,T)$ is positive,
there exists an entropy pair $(x_1,x_2)$ in $(X,T)$ \cite{B2}.
Let $U_i$ be a neighborhood of $x_i$ for $i=1,2$ with $d(U_1,U_2)>2\delta$.
Let $S$ be a positive density subset of $\N$ associated with $(U_1,U_2)$ in Proposition~\ref{thm:entropy-pair}.

Let $x\in X$ be a transitive point and $U$ be a neighborhood of $x$.
By Lemma~\ref{lem:rec-ip}, $N(x,U)$ contains an IP-set.
Then by Lemma~\ref{lem:pubd-ip} there are $q,q_1,q_2\in S$ and
$l_1,l_2\in N(x,U)$ with $l_1=q_1-q$, $l_2=q_2-q$ and
\[BD^*((S-q_1)\cap (S-q_2))= BD^*((S-l_1)\cap(S-l_2))\geq \frac 14(BD^*(S))^4.\]

For each $n\in (S-q_1)\cap (S-q_2)$, there exist  $s_{n,1}<s_{n,2}\in S$ such that
$n=s_{n,1}-q_1=s_{n,2}-q_2$.
If $n_1<n_2$, then $s_{n_1,1}<s_{n_2,1}$ and $s_{n_1,2}<s_{n_2,2}$.
So $\{s_{n_1,1},s_{n_1,2}\}\cap \{ s_{n_2,1},s_{n_2,2}\}\neq\emptyset$ if and only if $s_{n_1,2}=s_{n_2,1}$.
Then for every $n\in (S-q_1)\cap (S-q_2)$, there is at most one $m\in  (S-q_1)\cap (S-q_2)$
such that $\{s_{n,1},s_{n,2}\}\cap \{ s_{m,1},s_{m,2}\}\neq\emptyset$.
Let $H$ be a maximal subset of $(S-q_1)\cap (S-q_2)$ with the property that for every $m\neq n\in H$,
$\{s_{n,1},s_{n,2}\}\cap \{ s_{m,1},s_{m,2}\}=\emptyset.$
Then
\[BD^*(H)\geq \frac{1}{2} BD^*((S-q_1)\cap (S-q_2))\geq \frac{1}{8}(BD^*(S))^4.\]
Let $S_1=\{s_{n,1}: n\in H\}$ and $S_2=\{s_{n,2}: n\in H\}$.
Then $S_1\cup S_2\subset S$ and $S_1\cap S_2=\emptyset$.

Since $BD^*(H)\geq \frac{1}{8}(BD^*(S))^4$, we can find integers $M_j, N_j\in\N$ with $N_j-M_j\rightarrow\infty$
as $j\rightarrow\infty$, and
\[
\frac{|H\cap[M_j,N_j-1]|}{N_j-M_j}\geq \frac{1}{8}(BD^*(S))^4-\frac{1}{j}.
\]

Since $S$ is an independent set for $(U_1,U_2)$, for any $j\geq 1$ there is $y\in X$ with $T^qy\in U_1$,
$T^{q_i} y\in U_i$ for $i=1,2$ and
$T^{s_{n,i}}(y)\in U_i$ for $i=1,2$ and $n\in H\cap[M_j,N_j-1]$.
Then $H\cap [M_j,N_j-1]\subset N(T^{q_1}y,U_1)\cap N(T^{q_2}y,U_2)$.

Since $x$ is a transitive point, there is a sequence $\{m_i\}$
with $\lim_{i\to\infty}T^{m_i}x=T^q y$.
Then $T^{l_1}x, T^{l_2}x\in U$ and
\[\lim_{i\to\infty}T^{m_i}(T^{l_1}x)=T^{q_1}y\in U_1, \lim_{i\to\infty}T^{m_i}(T^{l_2}x)=T^{q_2}y\in U_2.\]

Since $T$ is continuous, we can find $\eta>0$ such that if $u,v\in X$ with $d(u,v)<\eta$ then $d(T^ku,T^kv)
<\frac{\delta}{2}$. Then there exists $N\in\N$, such that if $t\geq N$, then $$d(T^{m_t}(T^{l_1}x),T^{q_1}y)
<\eta,d(T^{m_t}(T^{l_2}x),T^{q_2}y)<\eta.$$ For such $m_t$, we have
\[
d(T^{m_t+k}(T^{l_1}x),T^{q_1+k}y)<\frac{\delta}{2},
d(T^{m_t+k}(T^{l_2}x),T^{q_2+k}y)<\frac{\delta}{2},
\]
if $k\in H\cap [M_j,N_j-1]$ since then $T^{q_1+k}y\in U_1$ and $T^{q_2+k}y\in U_2$. From the fact that $d(U_1,U_2)>2\delta$
we have that
\[
d(T^{m_t+k}(T^{l_1}x),T^{m_t+k}(T^{l_2}x))>\delta
\]
and so
\[
\frac{|\{n\in\Z: d(T^n(T^{l_1}x),T^n(T^{l_2}x))>\delta\}\cap[m_t+M_j,m_t+N_j-1]|}{N_j-M_j}\geq \frac{1}{8}(BD^*(S))^4-\frac{1}{j}.
\]

Since $j$ is arbitrary, we have that the set
\[\{n\in\Z: d(T^n(T^{l_1}x),T^n(T^{l_2}x))>\delta\}\]
has upper Banach density no less than $\frac{1}{8}(BD^*(S))^4$,
which implies that $(X,T)$ is Banach mean sensitive by Proposition \ref{property-6}(5).
\end{proof}

A direct consequence of Theorem \ref{thm:pe-bms} is:

\begin{cor}\label{entropyzero}
If $(X,T)$ is almost Banach mean equicontinuous then the topological entropy of $(X,T)$ is zero.
\end{cor}

To end the paper we discuss minimal examples which are Banach mean equicontinuous. It is easy to check
that the Denjoy example or Sturmian minimal systems are minimal Banach equicontinuous.
In~\cite{A59} Auslander gave an example which is minimal mean equicontinuous by modifying an example of Floyd.
We will refer this as the Auslander-Floyd example.
We can show that the Auslander-Floyd example in~\cite{A59} is also Banach mean equicontinuous.
In fact, we can prove a slightly more general result, and
it is easy to see that the Auslander-Floyd example satisfies the condition of the following proposition.
\begin{prop}
Let $(X,T)$ be a dynamical system. If for every $\ep>0$ there exists an open cover
$\{U_1,\dotsc,U_n\}$ of $X$ such that
\begin{enumerate}
  \item $T(U_i)\subset U_{i+1\pmod n}$ for $i=1,\dotsc,n$, and
  \item $\#\{i: \diam(U_i)\geq \varepsilon\}<\varepsilon n$,
\end{enumerate}
then $(X,T)$ is Banach mean equicontinuous.
\end{prop}
\begin{proof}
Let $\delta>0$ be a Lebesgue number of the open cover.
For every $x,y\in X$ with $d(x,y)<\delta$, there is $i_0\in\{1,\dotsc,n\}$ such that $x,y\in U_{i_0}$.
Without loss of generality, we assume that $i_0=1$.
For $0<N<M$ and $M-N>n$
\begin{align*}
  \frac{1}{M-N}\sum_{i=N}^{M-1}d(T^ix,T^iy)&\leq \frac{1}{M-N}\sum_{i=N}^{M-1}\diam(U_{i\pmod n})\\
  &<\varepsilon+\frac{\diam(X)}{M-N}\cdot\frac{M-N+n}{n}\cdot \varepsilon n<(1+2\diam(X))\varepsilon.
\end{align*}
Thus $(X,T)$ is Banach mean equicontinuous.
\end{proof}

\section{Further discussions}
By Corollary~\ref{cor:almost-mean-equi-entropy}, there are many  almost mean equicontinuous systems
which have positive entropy. Then by Theorem~\ref{thm:pe-bms} they are not almost Banach mean equicontinuous.
But the following question is still open.

\begin{ques} Is there a minimal system which is mean equicontinuous but not Banach mean equicontinuous?
\end{ques}

We know that a mean equicontinuous system is a proximal extension of its maximal equicontinuous factor.
It would be interesting to know whether the following question has a positive answer.
\begin{ques}
Is a minimal Banach equicontinuous system an almost 1-1 extension of its maximal equicontinuous factor?
\end{ques}

We remark that the same question is asked for minimal mean equicontinuous systems in \cite{A59}.

\medskip
\noindent {\bf Acknowledgments.}
The first author was supported in part by Scientific Research Fund of Shantou University (YR13001),
Guangdong Natural Science Foundation (S2013040014084) and NNSF of China (11326135).
The second and third authors were supported in part by NNSF of China (11371339).
The authors would like to thank Wen Huang, Song Shao for very useful discussions; and Felipe Garc\'ia-Ramos, Benjy Weiss for very useful comments.
The authors would also like to thank the anonymous referee for his/her helpful suggestions concerning this paper.

\bibliographystyle{amsplain}

\begin{thebibliography}{99}
\bibitem{AAB} E. Akin, J. Auslander and K. Berg, \emph{When is a transitive map chaotic?},
Convergence in ergodic theory and probability (Columbus, OH, 1993), 25--40,
Ohio State Univ. Math. Res. Inst. Publ., 5, de Gruyter, Berlin, 1996.

\bibitem{A59} J. Auslander, \emph{Mean-$L$-stable systems},
Illinois J. Math., {\bf 3} (1959), 566--579.

\bibitem{A} J. Auslander, \emph{Minimal Flows and Their Extensions}, North-Holland Publishing Co., Amsterdam, 1988.

\bibitem{AY80} J. Auslander and J. A. Yorke, \emph{Interval maps, factors of maps, and chaos},
Tohoku Math. J., \textbf{32} (1980), no. 2, 177--188.

\bibitem{B2} F. Blanchard, \emph{A disjointness theorem involving
topological entropy}, Bull. de la Soc. Math. de France, {\bf
121}(1993), 465--478.

\bibitem{EG60} R. Ellis and W. H. Gottschalk, \emph{Homomorphisms of transformation groups}, Trans. Amer.
Math. Soc. \textbf{94} (1960), 258--271.

\bibitem{tomasz} T. Downarowicz, \emph{Positive topological entropy implies chaos DC2},
Proc. Amer. Math. Soc., \textbf{142} (2014),  137--149.

\bibitem{F51} S. Fomin, \emph{On dynamical systems with a purely point spectrum}, Dokl. Akad. Nauk SSSR,
vol. {\bf 77} (1951), 29--32 (In Russian).


\bibitem{F81} H. Furstenberg, \emph{Recurrence in Ergodic Theory and Combinatorial Number Theory},
Princeton Univ. Press, Princeton, NJ, 1981.

\bibitem{Felipe} F. Garc\'ia-Ramos, \emph{Weak forms of topological and measure theoretical equicontinuity:
relationships with discrete spectrum and sequence entropy}, arXiv:1402.7327v2[math.DS].

\bibitem{Gi36}  J. Gillis, \emph{Notes on a property of measurable sets}, J. London Math. Soc., \textbf{11} (1936), 139--141.

\bibitem{G07}E. Glasner, \emph{The structure of tame minimal dynamical systems},
Ergod. Th. \& Dynam. Sys., \textbf{27} (2007), 1819--1837.

\bibitem{GM89} S. Glasner and D. Maon, \emph{Rigidity in topological dynamics},
Ergod. Th. \& Dynam. Sys., \textbf{9} (1989), 309--320.

\bibitem{GMU} E. Glasner, M. Megrelishvili and V. Uspenskij, \emph{On metrizable enveloping semigroups}.
Israel J. Math., {\bf 164} (2008), 317--332.

\bibitem{GW93} E.~Glasner and B.~Weiss, \emph{Sensitive dependence on initial conditions},
Nonlinearity, \textbf{6} (1993), no. 6, 1067--1075.

\bibitem{Ha67} F. Hahn, Y. Katznelson, \emph{On the entropy of uniquely ergodic transformations},
Trans. Amer. Math. Soc., {\bf 126} (1967), 335--360.

\bibitem{HLSY03} W. Huang, S. Li, S. Shao and X. Ye,
\emph{Null systems and sequence entropy pairs}, Ergod. Th. \& Dynam. Sys., \textbf{23} (2003), 1505--1523.

\bibitem{HLY} W. Huang, P. Lu and X. Ye, \emph{Measure-theoretical sensitivity and equicontinuity}, Israel J. of
Math., {\bf 183} (2011),  233--283.

\bibitem{HLY13} W. Huang, J. Li and X. Ye, \emph{Stable sets and mean Li-Yore chaos
in positive entropy systems}, Journal of Functional Analysis, \textbf{266} (2014), 3377--3394.

\bibitem{HY02} W. Huang and X. Ye, \emph{Devaney¡¯s chaos or 2-scattering implies Li-Yorke¡¯s chaos},
Topol. Appl., \textbf{117} (2002), 259--272.

\bibitem{HY} W. Huang  and X. Ye, \emph{Minimal sets in
almost equicontinuous systems}, Proc. of the Steklov Inst. of Math.,
{\bf 244} (2004), 280--287.

\bibitem{HY06} W. Huang and X.~Ye, \emph{A local variational relation and applications},
Israel J. Math., \textbf{151} (2006) 237--280.

\bibitem{KL07} D. Kerr and H. Li, \emph{Independence in topological and C$^*$-dynamics},
Math. Ann., \textbf{338} (2007),  869--926.

\bibitem{LT13} J. Li and S. Tu, \emph{On proximality with Banach density one},
J. Math.Anal.Appl. \textbf{416} (2014), 36--51.

\bibitem{OW} D. Ornstein and B. Weiss, \emph{Mean distality and tightness},
Proc. Steklov Inst. Math., \textbf{244} (2004), no.1, 295--302.

\bibitem{O52} J. C. Oxtoby, \emph{Ergodic sets}, Bull. Amer. Math. Soc., {\bf 58} (1952), 116--136.

\bibitem{S82} B. Scarpellini,
\emph{Stability properties of flows with pure point spectrum},
J. London Math. Soc. (2)  {\bf 26} (1982), no. 3, 451--464.

\bibitem{W82} P. Walters, \emph{An introduction to ergodic theory}.
Graduate Texts in Mathematics, 79. Springer-Verlag,
New York-Berlin, 1982.

\bibitem{YZ} X. Ye and R. Zhang, \emph{On sensitivity sets
in topological dynamics}, Nonlinearity, {\bf 21} (2008), 1601--1620.

\end{thebibliography}

\end{document}